\documentclass{amsart}
\usepackage[english]{babel}
\usepackage[latin1]{inputenc}
\usepackage[dvips,final]{graphics}
\usepackage{amsmath,amsfonts,amssymb,amsthm,amscd,array,stmaryrd,mathrsfs, mathdots, epigraph}
\usepackage{arydshln}
\usepackage[makeroom]{cancel}
\usepackage{pstricks}
 \usepackage[all]{xy}
 \usepackage{url}
\usepackage{multirow, blkarray}
\usepackage{booktabs}
\usepackage{textcomp}
 \usepackage[final]{epsfig}
 \usepackage{color}
\theoremstyle{plain}
\newtheorem{thm}{Theorem}[section]

\newtheorem{lem}[thm]{Lemma}
\newtheorem{cor}[thm]{Corollary}
\newtheorem{prop}[thm]{Proposition}
\newtheorem{defn}[thm]{Definition}
\newtheorem{pro}[thm]{Problem}

\theoremstyle{definition}
\newtheorem{rem}[thm]{Remark}

\begin{document}

\title{A dense subset of  $M_{n}(\mathbb{R})$ containing diagonalizable matrices}
\date{}
\author{Flavien Mabilat$^{a}$}
\address{$^{a}$,
Laboratoire de Math\'{e}matiques de Reims,
UMR9008 CNRS et Universit\'{e} de Reims Champagne-Ardenne, 
U.F.R. Sciences Exactes et Naturelles 
Moulin de la Housse - BP 1039 
51687 Reims cedex 2,
France}
\email{flavien.mabilat@univ-reims.fr}

\maketitle

\begin{abstract}

In this note, we consider matrices similar to $X$-form matrices, which are the matrices for which only the diagonal and the anti-diagonal elements can be different from zero. First, we give a characterization of these matrices using the minimal polynomial. Then,  we prove that the set of matrices similar to $X$-form matrices over $\mathbb{R}$ and $\mathbb{C}$ are dense and we give a characterization of the interior of this set.
\\

\end{abstract}
\thispagestyle{empty}

\\
\\Declarations of interest: none
\\
\\
\begin{flushright}
 \textit{``La m\'{e}moire est aussi menteuse que l'imagination, et bien 
\\plus dangereuse avec ses petits airs studieux.''} 
\\ Fran\c{c}oise Sagan, \textit{Derri\`{e}re l'\'{e}paule}
\end{flushright}

\section{Introduction}

In this note, all fields considered are commutative. Let $\mathbb{K}$ be an arbitrary commutative field. The set of all square matrices of size $n$ over $\mathbb{K}$ is denoted $M_{n}(\mathbb{K})$, $0_{n}$ is the zero matrix of $M_{n}(\mathbb{K})$. If $A \in M_{n}(\mathbb{K})$ we denote $\pi_{A}(X)$ the minimal polynomial of $A$ (with the convention $\pi_{A}$ monic polynomial) and $\chi_{A}(X)={\rm det}(XI_{n}-A)$ the characteristic polynomial of $A$ (with this definition $\chi_{A}(X)$ is a monic polynomial). We use the convention $\prod_{i=1}^{0} a_{i} =1$. An elementary Jordan matrix is a matrix composed of zeroes everywhere except for the diagonal, which is filled with a fixed element $\lambda \in \mathbb{K}$, and for the superdiagonal, which is composed of ones. The Frobenius companion matrix of the monic polynomial $P(X)=X^{n}+\sum_{i=0}^{n-1} a_{i}X^{i}$ is the square matrix defined as $C(P)=\begin{pmatrix}
            0 & \cdots & 0 & -a_{0} \\
					  1 &  \ddots & \vdots & -a_{1} \\
						  & \ddots & 0 & \vdots\\
						  &        & 1 & -a_{d-1} \\
							\end{pmatrix}.$
\\ 
\\ \indent A classical results states that the set of diagonalizable matrices is dense over $\mathbb{C}$ but not over $\mathbb{R}$. Here, we want to find a subset of $M_{n}(\mathbb{K})$ containing all diagonalizables matrices which is dense in the cases $\mathbb{K}=\mathbb{R}, \mathbb{C}$ and we want to find the interior of his set. For this, we will study the objects introduced in the following definitions :

\begin{defn}
\label{def}

i) A $X$-form matrix is a square matrix for which only the diagonal and the anti-diagonal elements can be different from zero.
\\ii) An endomorphism $u$ of a finite dimensional vector space $E$ is $X$-formable if there is a basis of $E$ with respect to which the matrix of $u$ is a $X$-form matrix.
\\iii) A $X$-formable matrix is a matrix similar to a $X$-form matrix.

\end{defn}

\noindent First, we give some of the easy properties verifying by that kind of matrices :
\begin{itemize}
\item Matrices of size 1 or 2 are $X$-form matrices;
\item $X$-form matrices of odd size has an eigenvalue belonging to $\mathbb{K}$;
\item If $A$ is a $X$-form matrix then $A^{t}$ is a $X$-form matrix;
\item If $A$ is diagonalizable then $A$ is $X$-formable;
\item The set of $X$-form matrices of size $n$ is a vector subspace of $M_{n}(\mathbb{K})$;
\item If $A$ and $B$ are $X$-form matrices then $AB$ is a $X$-form matrix;
\item If $A$ is an invertible $X$-form matrix then $A^{-1}$ is a $X$-form matrix (this follows from the last two points and the equality $A^{-1}=\frac{-1}{a_{0}}(A^{r-1}+\sum_{i=1}^{r-1}a_{i} A^{i-1})$ with $\pi_{A}(X)=X^{r}+\sum_{i=0}^{r-1}a_{i} X^{i}$, $a_{0} \neq 0$ since $A$ is invertible).
\\
\end{itemize}

We also have some other properties in the case of central-symmetric $X$-form matrices (see \cite{ANO,S}) and also in the case of block central-symmetric $X$-form matrices (see \cite{Sa}). 
\\
\\ \indent A classical result states that a square matrix is diagonalizable if and only if its minimal polynomial is a product of distinct linear factors over $\mathbb{K}$. It is natural to find a similar characterization in the case of $X$-formable matrices. This seems to be not very difficult but we have been unable to locate such a result in the litterature. Hence, our first objective is to prove the following result :

\begin{thm}
\label{MaThm}

Let $\mathbb{K}$ be a commutative field and $A \in M_{n}(\mathbb{K})$. $A$ is $X$-formable if and only if \[\pi_{A}(X)=\prod_{i=1}^{r} P_{i}(X) \prod_{i=1}^{q} (X-\lambda_{i})^{n_{i}},\] with $r, q \geq 0$, $P_{i}$ irreducible monic polynomials of degree 2, $\lambda_{i} \in \mathbb{K}$, $1 \leq n_{i} \leq 2$, $P_{i} \neq P_{j}$ and $\lambda_{i} \neq \lambda_{j}$ for $i \neq j$.

\end{thm}

To simplify the proof, we will use the following notation : a polynomial verifies the property $(\mathcal{P})$ if it has the same factorization as in the previous theorem.
\\
\\ \indent Then, we will consider some topological properties of $X$-formable matrices over $\mathbb{R}$ and $\mathbb{C}$ related to our initial goal. The main results of this text gathered in the following theorem give an answer to this problem :

\begin{thm}
\label{MaThmBis}

$n \in \mathbb{N}^{*}$. Let $F$ be the set of $X$-formable matrices of size $n$ over $\mathbb{C}$ and $G$ the set of $X$-formable matrices of size $n$ over $\mathbb{R}$. 
\\ i) $F$ is dense in $M_{n}(\mathbb{C})$.
\\ ii) The interior of $F$ is $\{A \in M_{n}(\mathbb{C}),~\chi_{A}~{\rm verifies}~(\mathcal{P})\}=\{A \in M_{n}(\mathbb{C}),~\chi_{A}~{\rm has~only~simple~or~double~roots}\}$.
\\ iii) $G$ is dense in $M_{n}(\mathbb{R})$.
\\ iv) The interior of $G$ is $\{A \in M_{n}(\mathbb{R}),~\chi_{A}~{\rm verifies}~(\mathcal{P})\}$.

\end{thm}

\section{$X$-formable matrices}

\noindent In this subpart, $\mathbb{K}$ is an arbitrary commutative field. We begin by an easy lemma.

\begin{lem}
\label{21}

A matrix $A$ is $X$-formable if and only if it is similar to a block-diagonal matrix in which each block has a size less than 2.

\end{lem}

\begin{proof}

First, we consider the case $n=2m$. Let $A$ be a $X$-formable matrix of size $n$. Let $u$ be an endomorphism of $\mathbb{K}^{n}$ whose matrix in the canonical basis is $A$. It exists $\mathcal{B}=(e_{1},\ldots,e_{2m})$ a basis of $\mathbb{K}^{n}$ such that the matrix of $u$ in $\mathcal{B}$ is a $X$-form matrix. The matrix of $u$ in the basis $(e_{1},e_{2m},e_{2},e_{2m-1},\ldots,e_{m},e_{m+1})$ is a block-diagonal matrix in which each block has a size less than 2. Let $A$ be a matrix similar to a block-diagonal matrix in which each block has a size less than 2 and $u$ be an endomorphism of $\mathbb{K}^{n}$ whose matrix in the canonical basis is $A$. It exists $\mathcal{B'}=(f_{1},\ldots,f_{2m})$ a basis of $\mathbb{K}^{n}$ such that the matrix of $u$ in $\mathcal{B'}$ is a block-diagonal matrix in which each block has a size less than 2. The matrix of $u$ in the basis $(f_{1},f_{3},\ldots,f_{4},f_{2})$ is a $X$-form matrix. The proof of the case $n=2m+1$ is similar.

\end{proof}

Note that this result, combined with the normal matrices reduction theorem (see \cite{G} Theorem 4, p 271), allows us to see that normal matrices on $\mathbb{R}$ are $X$-formable.

\subsection{Proof of theorem \ref{MaThm}}

\begin{proof}

Let $A$ be a $X$-formable matrix of size $n$. By the previous lemma, it exists a basis of $\mathbb{K}^{n}$ such that the matrix of $u$ in this basis is a block-diagonal matrix $C$ in which each block has a size less than 2. Thus, $\pi_{A}$ is the least common multiple of the minimal polynomials of the diagonal blocks of C. The minimal polynomials of matrix of size 1 or 2 are one of the following type :
\begin{itemize}
\item irreducible monic polynomial of degree 2;
\item $(X-\lambda)^{2}$;
\item $(X-\lambda)$;
\item $(X-\lambda_{1})(X-\lambda_{2})$, $\lambda_{1} \neq \lambda_{2}$.
\end{itemize}

\noindent Hence, they verify $(\mathcal{P})$. Thus, $\pi_{A}$ verifies $(\mathcal{P})$.
\\	
\\Now, we want to prove the remaining implication of the theorem.
\\
\\Firt, we consider a matrix $A$ of size $n$ such that $\pi_{A}$ verifies $(\mathcal{P})$ and such that $\pi_{A}$ has only one irreducible factor (over $\mathbb{K}$). We have three cases :

\begin{itemize}
\item $\pi_{A}=(X-\lambda)$  then $A$ is diagonalizable, and so $X$-formable. 
\item $\pi_{A}=(X-\lambda)^{2}$. By the decomposition theorem of Jordan, $A$ is similar to a block diagonal matrix $H=\begin{pmatrix}
   J_{1} &        &   \\
         & \ddots &    \\
		     &        & J_{l}  \\
\end{pmatrix}$ in which $J_{i}$ is an elementary Jordan matrix. Besides, each $J_{i}$ is a square matrix of size $n_{i} \in \{1, 2\}$. Indeed, $\pi_{A}=\pi_{H}={\rm lcm}(\pi_{J_{i}}, 1 \leq i \leq l)$ and $\pi_{J_{i}}=(X-\lambda)^{n_{i}}$. Hence, $A$ is $X$-formable, by lemma \ref{21}.
\item $\pi_{A}$ is an irreducible monic polynomial of degree 2. By the decomposition theorem of Frobenius, $A$ is similar to a block diagonal matrix \[H=\begin{pmatrix}
   C(R_{1}) &        &   \\
         & \ddots &    \\
		     &        & C(R_{t})  \\
\end{pmatrix},\] with $R_{i}$ monic polynomials verifying $R_{i}$ divides $R_{i+1}$. Since $\pi_{A}=\pi_{H}={\rm lcm}(R_{i}, 1 \leq i \leq t)$ and $\pi_{A}$ irreducible, $R_{i}=\pi_{A}$. Thus, $H$ is block-diagonal matrix in which each block has size 2 and $A$ is $X$-formable, by lemma \ref{21}.
\\
\end{itemize}

\noindent Now, we proceed by induction on the size $n$ of the matrix. If $n=1$, then the result is true. Suppose it exists $n \geq 1$ such that all matrices of size less than $n$ whose minimal polynomial verifies $(\mathcal{P})$ are $X$-formable. Let $A$ be a matrix of size $n+1$ whose minimal polynomial verifies $(\mathcal{P})$ and $u$ the endomorphism canoniquely associated to $A$. If $\pi_{A}$ has only one irreducible factor, then the result is true by the previous discussion. Suppose $\pi_{A}$ has several irreducible factors. It exists a  monic polynomial $P$ such that $\pi_{A}=P\frac{\pi_{A}}{P}$, $P$ has degree greater or equal to 1, $P$ and $\frac{\pi_{A}}{P}$ are relatively prime polynomials. $P$ and $\frac{\pi_{A}}{P}$ verify $(\mathcal{P})$. Consider $F={\rm Ker}(P(u))$ and $G={\rm Ker}(\frac{\pi_{A}}{P}(u))$. Hence, by the kernel lemma\footnote{This translation of the French name "lemme des noyaux" seems to be the most used name for this result.}, $\mathbb{K}^{n+1}=F \oplus G$ and $F$, $G$ are invariant subspaces of $u$. 
\\
\\By induction assumption, $u_{|G}$ (restriction of $u$ to $G$) and $u_{|F}$ are $X$-formable. Hence, $u$ is $X$-formable and theorem~\ref{MaThm} is proved.

\end{proof}

\medskip

\subsection{Some additional elements about theorem \ref{MaThm}}

\leavevmode\par 
\leavevmode\par \noindent In the case of an algebraically closed field we have the following result :

\begin{cor}
\label{22}

Let $\mathbb{K}$ be an algebraically closed field and $A \in M_{n}(\mathbb{K})$. The following assertions are equivalent :
\\i) $A$ is $X$-formable.
\\ii) $\pi_{A}(X)=\prod_{i=1}^{q} (X-\lambda_{i})^{n_{i}}$ with $q \geq 1$, $\lambda_{i} \in \mathbb{K}$, $1 \leq n_{i} \leq 2$ and $\lambda_{i} \neq \lambda_{j}$ for $i \neq j$.
\\iii) All the Jordan blocks appearing in the Jordan normal form of $A$ have their size equal to 1 or 2.

\end{cor} 

\noindent If $A$ is a $X$-formable matrix then there is not a unique $X$-form matrix similar to $A$. For example, $A=\begin{pmatrix}
    1 & 1  \\[4pt]
    0 & 2   \\
 \end{pmatrix}$ is a $X$-form matrix but $A$ is also similar to $\begin{pmatrix}
    1 & 0  \\[4pt]
    0 & 2   \\
 \end{pmatrix}$ which is also a $X$-form matrix.
\\
\\We now give the two following examples :
\begin{itemize}
\item We consider $\mathbb{K}=\mathbb{R}$ and $B=\begin{pmatrix}
    1 & 1 & 1 \\
    0 & 1 & 1  \\
		 0 & 0 & 1 \\
   \end{pmatrix}$. $\pi_{B}(X)=(X-1)^{3}$. Hence, by the theorem \ref{MaThm}, $B$ is not $X$-formable.
\\

\item We consider $\mathbb{K}=\mathbb{R}$ and $C=\begin{pmatrix}
    -22 & 47 & -19 & 18 \\
    1 & 3 & -3 & -5 \\
		 14 & -23 & 7 & -16\\
		-15 & 27 & -9 & 17 \\
   \end{pmatrix}$. $\pi_{C}(X)=(X-2)(X-3)(X^{2}+2)$. Hence, by the theorem \ref{MaThm}, $C$ is $X$-formable. For instance,
	\[C=\begin{pmatrix}
    1 & 2 & 3 & 4 \\
    0 & 2 & 1 & 2 \\
		 0 & 3 & 0 & 1\\
		1 & 0 & 1 & 1 \\
   \end{pmatrix} \begin{pmatrix}
    4 & 0 & 0 & 1 \\
    0 & 1 & 1 & 0 \\
		 0 & -3 & -1 & 0\\
		-2 & 0 & 0 & 1 \\
   \end{pmatrix} \begin{pmatrix}
    -1 & 1 & 0 & 2 \\
    1 & -2 & 1 & -1 \\
		 4 & -7 & 2 & -4\\
		-3 & 6 & -2 & 3 \\
   \end{pmatrix}.\]
	
\end{itemize}

We conclude this part by giving some elements about the product of $X$-formable matrices. We have the following result :

\begin{thm}[Botha, \cite {B} Theorem 2.1]

Let $\mathbb{K}$ be any field such that the characteristic of $\mathbb{K}$ is different from 2 and such that $\mathbb{K}\neq \mathbb{F}_{3}$ (the field with 3 elements). Then every matrix over $\mathbb{K}$ is a product of two diagonalizable matrices.

\end{thm}

Since diagonalizable matrices are $X$-formable, this result covers a lot of cases. Here, we consider the case of $\mathbb{F}_{3}$. By the decomposition theorem of Frobenius, it is sufficient to consider companion matrices. We have the following equality :

\[\begin{pmatrix}
            0 & \cdots & 0 & a_{0} \\
					  1 &  \ddots & \vdots & a_{1} \\
						  & \ddots & 0 & \vdots\\
						  &        & 1 & a_{d-1} \\
							\end{pmatrix}=\begin{pmatrix}
					     &         & 1 \\
						   & \iddots & \\
						 1 &         &  \\
							\end{pmatrix}\begin{pmatrix}
               &         & 1 & a_{d-1} \\
						   & \iddots &  & \vdots \\
						 1 &         & & a_{1} \\
						 0 & \ldots & 0 & a_{0}\\
							\end{pmatrix}.\]
\noindent The first matrix on the right is diagonalizable since its minimal polynomial is $(X-1)(X+1)$ and $1 \neq -1$ in $\mathbb{F}_{3}$. If $a_{0}=\pm 1$ then the second matrix on the right is $X$-formable since its minimal polynomial divides $(X-1)^{2}(X+1)^{2}$ (theorem \ref{MaThm}). If $a_{0}=0$ then the second matrix on the right is $X$-formable since its minimal polynomial divides $X^{2}(X-1)^{2}(X+1)^{2}$ (theorem \ref{MaThm}). Hence, every matrix over $\mathbb{F}_{3}$ is a product of two $X$-formable matrices.					
	
\section{Some topological aspects of the set of $X$-formable matrices}

The aim of this section is to prove theorem \ref{MaThmBis}. Here, we suppose $\mathbb{K}=\mathbb{R}, \mathbb{C}$. The set of polynomials of degree less than $n$ over $\mathbb{K}$ is denoted $\mathbb{K}_{n}[X]$. We use the following norm, if $A=(a_{i,j})_{1 \leq i,j \leq n} \in M_{n}(\mathbb{K})$ we denote $\left\| A\right\|_{\infty}={\rm max}(\left|a_{i,j}\right|,1 \leq i,j \leq n)$ (we recall that all norms are equivalent on $M_{n}(\mathbb{K})$). We start by some preliminary results.

\subsection{Prelimirary lemmas}

In this subsection, we give some results concerning polynomials over $\mathbb{K}$. If $P(X)=\sum_{k=0}^{n} a_{k}X^{k}$, we denote $\left\| P\right\|_{\infty}={\rm max}(\left|a_{k}\right|,~0 \leq k \leq n)$ and $D(z,r)=\{w \in \mathbb{C},~\left|z-w\right|<r\}$. We begin by the following well-known result :

\begin{thm}[Continuity of the roots of a polynomial, \cite{L}]
\label{31}

Let $P$ be a polynomial, $z_{i}$ its distinct roots with $i=1,\ldots,p$, $m_{i}$ the multiplicity of the root $z_{i}$ ($m_{1}+...+m_ {p}=deg(P)$). Then for any $\epsilon >0$ such that $D(z_{i},\epsilon) \cap D(z_{j},\epsilon)=\emptyset$ for any $i \neq j$ there exists $\eta >0$ such that any polynomial $Q$ whose coefficients differ from those of $P$ only by less than $\eta$ has exactly $m_{i}$ roots (distinct or not) in $D(z_{i},\epsilon)$.

\end{thm}

\noindent Now we can prove the following result :

\begin{prop}
\label{32}

$n \in \mathbb{N}^{*}$. The set $\hat{F}$ of polynomials over $\mathbb{C}$ of degree $n$ whose roots are simple or double is an open subset of $\mathbb{C}_{n}[X]$.

\end{prop}

\begin{proof}

Let $P \in \mathbb{C}_{n}[X]$ with ${\rm deg}(P)=n$. We assume that P has only simple and double roots. We denote $z_{i}$ its distinct roots with $i=1,\ldots,p$ and $m_{i}$ the multiplicity of the root $z_{i}$. We have $m_{i} \leq 2$ and $m_{1}+...+m_ {p}=n$. 
\\
\\Since $\mathbb{C}$ is a separated nomed space, it exists $\epsilon >0$ such that $D(z_{i},\epsilon) \cap D(z_{j},\epsilon)=\emptyset$ for any $i \neq j$. Hence, by the previous theorem it exists $\eta >0$ such that any polynomial $Q$ verifying $\left\| P-Q\right\|_{\infty} < \eta$ has exactly $m_{i}$ roots (distinct or not) in $D(z_{i},\epsilon)$. So, in each $D(z_{i},\epsilon)$, $Q$ has exactly one root of multiplicity less than two or two roots of multiplicity one. Besides, $Q$ has $n$ roots (with multiplicity) in $\bigcup_{i=1}^{p} D(z_{i},\epsilon)$ and ${\rm deg}(Q) \geq n$.
\\
\\ Hence, a polynomial $T$ belonging to $\mathbb{C}_{n}[X]$ and verifying $\left\| P-T\right\|_{\infty} < \eta$ has degree $n$ and all its roots are simple or double. Thus, $\hat{F}$ is an open subset of $\mathbb{C}_{n}[X]$.

\end{proof}

This result is no longer true if we replace $\mathbb{C}$ by $\mathbb{R}$. For instance, we can consider the polynomial sequence $P_{n}=X^{2}+\frac{1}{n}X+\frac{1}{n}$. For all $n \in \mathbb{N}^{*}$, $P_{n}$ is irreducible, since its discriminant is $\frac{1}{n}(\frac{1}{n}-4)<0$. However, this sequence converges to $X^{2}$. Hence, the complementary of the set of polynomials over $\mathbb{R}$ of degree $2$ whose roots are simple or double is not closed. So, this set is not an open set. However, we can still have a similar result :

\begin{prop}
\label{33}

$n \in \mathbb{N}^{*}$. The set $\hat{G}$ of polynomials over $\mathbb{R}$ of degree $n$ which verify $(\mathcal{P})$ is an open subset of $\mathbb{R}_{n}[X]$.

\end{prop}

\begin{proof}

Let $P \in \mathbb{R}_{n}[X]$ with ${\rm deg}(P)=n$. We assume that $P$ verifies $(\mathcal{P})$. We denote $z_{i}$ its distinct roots in $\mathbb{C}$ with $i=1,\ldots,p$ and $m_{i}$ the multiplicity of the root $z_{i}$. Since $P$ verify $(\mathcal{P})$, we have $m_{i} \leq 2$ and $m_{1}+...+m_ {p}=n$. Besides, if $z_{i} \in \mathbb{C}-\mathbb{R}$ then $m_{i}=1$ and $\overline{z_{i}}$ is also a root of $P$ of multiplicity 1. 
\\
\\Since $\mathbb{C}$ is a separated nomed space, it exists $\epsilon >0$ such that $D(z_{i},\epsilon) \cap D(z_{j},\epsilon)=\emptyset$ for any $i \neq j$. Hence, by the previous theorem it exists $\eta >0$ such that any polynomial $Q$ in $\mathbb{R}[X] \subset \mathbb{C}[X]$ verifying $\left\| P-Q\right\|_{\infty} < \eta$ has exactly $m_{i}$ roots (distinct or not) in $D(z_{i},\epsilon)$. In particular, $Q$ has $n$ roots in $\mathbb{C}$ (with multiplicity) belonging to $\bigcup_{i=1}^{p} D(z_{i},\epsilon)$ and ${\rm deg}(Q) \geq n$. We have several possible cases for the roots of $Q$ belonging to $D(z_{i},\epsilon)$ :
\begin{itemize}
\item $D(z_{i},\epsilon)$ contains exactly one real root of multiplicity less than two or two real roots of multiplicity one.
\\
\item $D(z_{i},\epsilon)$ contains exactly one root $\lambda$ of multiplicity one belonging to $\mathbb{C}-\mathbb{R}$. Then, $D(\overline{z_{i}},\epsilon)$ contains exactly one root $\overline{\lambda}$ of multiplicity one. Thus, the polynomial $(X-\lambda)(X-\overline{\lambda})$ is an irreducible real factor of $Q$ and this factor appears only one time in $Q$.
\\
\item $D(z_{i},\epsilon)$ contains exactly two roots of multiplicity one, $\lambda$ and $\mu$, belonging to $\mathbb{C}-\mathbb{R}$. In this case, $z_{i} \in \mathbb{R}$ (since the non-real roots of $P$ has multiplicity one). Besides, $\overline{\lambda}$ and $\overline{\mu}$ are roots of $Q$ of multiplicity one and they also belong to $D(z_{i},\epsilon)$ (since $z_{i} \in \mathbb{R}$). Hence, $\mu=\overline{\lambda}$. Thus, the polynomial $(X-\lambda)(X-\overline{\lambda})$ is an irreducible real factors of $Q$ and this factor appears only one time in $Q$.
\\
\item $D(z_{i},\epsilon)$ contains exactly one real root of multiplicity one $x$ and one root of multiplicity one $\lambda$ belonging to $\mathbb{C}-\mathbb{R}$. In this case, $z_{i} \in \mathbb{R}$ (since the non-real roots of $P$ has multiplicity one). Besides, $\overline{\lambda}$ is a root of $Q$ of multiplicity one which belongs to $D(z_{i},\epsilon)$ (since $z_{i} \in \mathbb{R}$). Hence, this case is not possible.
\\
\end{itemize}

\noindent Hence, a polynomial $T$ belonging to $\mathbb{R}_{n}[X]$ and verifying $\left\| P-T\right\|_{\infty} < \eta$  has degree $n$ and verifies $(\mathcal{P})$.  Thus, $\hat{G}$ is an open subset of $\mathbb{R}_{n}[X]$.

\end{proof}

\begin{rem}

{\rm In the first case considered in the proof, $z_{i} \in \mathbb{R}$. Indeed, suppose $z_{i}$ is a non-real root of $P$. $D(z_{i},\epsilon)$ contains only one root of multiplicity one (since the non-real roots of $P$ has multiplicity one). $\overline{z_{i}}$ is a root of $P$ of multiplicity one. Hence, $D(\overline{z_{i}},\epsilon)$ contains exactly one root of $Q$. We set $x$ this root. $x$ is necessarily real since otherwise $D(z_{i},\epsilon)$ would contain a non-real root which would be $\overline{x}$. Besides, $\left|\overline{z_{i}}-x\right|=\left|z_{i}-x\right| \leq \epsilon$. Thus, $x \in D(z_{i},\epsilon) \cap D(\overline{z_{i}},\epsilon)$. This is not possible since $D(z_{i},\epsilon) \cap D(\overline{z_{i}},\epsilon)= \emptyset$. Thus, $z_{i} \in \mathbb{R}$.
}

\end{rem}

\subsection{Proof of theorem \ref{MaThmBis}}

\begin{proof}

We consider the following continued map $\begin{array}{ccccc} 
\varphi & : & M_{n}(\mathbb{K}) & \longrightarrow & \mathbb{K}_{n}[X] \\
 & & A & \longmapsto & \chi_{A}  \\
\end{array}$.
\\
\\ i) The set of diagonalizable matrices $D_{n}(\mathbb{C})$ is included in $F$. Moreover, $D_{n}(\mathbb{C})$ is dense in $M_{n}(\mathbb{C})$. Hence, $F$ is dense in $M_{n}(\mathbb{C})$.
\\
\\ii) $\hat{F}$ is an open subset of $\mathbb{C}_{n}[X]$ (proposition \ref{32}). Hence, \[\varphi^{-1}(\hat{F})=\{A \in M_{n}(\mathbb{C}),~\chi_{A}~{\rm verifies}~(\mathcal{P})\}=\{A \in M_{n}(\mathbb{C}),~\chi_{A}~{\rm has~only~simple~or~double~roots}\}\] is an open subset of $M_{n}(\mathbb{C})$. Besides, $\varphi^{-1}(\hat{F})$ is included in $F$ (by the theorem \ref{MaThm} and the theorem of Cayley-Hamilton). Hence, $\varphi^{-1}(\hat{F})$ is included in the interior of $F$.
\\
\\Let $A \in F$ such that $\chi_{A}$ doesn't verify $(\mathcal{P})$. By corollary \ref{22}, all the Jordan blocks appearing in the Jordan normal form of $A$ have their size equal to 1 or 2. Since, $\chi_{A}$ doesn't verify $(\mathcal{P})$, it exists $\lambda$ such that $(X-\lambda)^{3}$ divides $\chi_{A}$. Thus, one of the following occurs :
\begin{itemize}
\item a) the Jordan normal form of $A$ contains two blocks $J_{\lambda}=\begin{pmatrix}
    \lambda & 1  \\
    0 & \lambda  \\
   \end{pmatrix}$;
\item b) the Jordan normal form of $A$ contains three blocks $(\lambda)$;
\item c) the Jordan normal form of $A$ contains one block $J_{\lambda}=\begin{pmatrix}
    \lambda & 1  \\
    0 & \lambda  \\
   \end{pmatrix}$ and one block $(\lambda)$.
\end{itemize}

\noindent We consider each case separately :
\begin{itemize}
\item If a) occurs. It exists $P \in GL_{n}(\mathbb{C})$ such that $A=P\begin{pmatrix}
    J_{\lambda} & &  \\
     & J_{\lambda} & \\
		  & & B	
			\\
   \end{pmatrix}P^{-1}$. For all $n \in \mathbb{N}^{*}$, we define $A_{n}=P\begin{pmatrix}
    \lambda & 1 & \frac{1}{n} & 0 & \\
    0 & \lambda & 0 & \frac{1}{n} & \\
		0 & 0 & \lambda & 1 & \\
		0 & 0 & 0 & \lambda & \\
		  &  &  & & B	\\
   \end{pmatrix}P^{-1}$. $(A_{n})$ converges to $A$. However, for all $n \in \mathbb{N}^{*}$, $(X-\lambda)^{3}$ divides $\pi_{A_{n}}$. Hence, $A_{n} \notin F$ (theorem \ref{MaThm}). Thus, $A$ doesn't belong to the interior of $F$.
\\
\item If b) occurs. It exists $P \in GL_{n}(\mathbb{C})$ such that $A=P\begin{pmatrix}
     \lambda & & & \\
     & \lambda & & \\
		 & & \lambda & \\
		  & & & B	
			\\
   \end{pmatrix}P^{-1}$. For all $n \in \mathbb{N}^{*}$, we define $A_{n}=P\begin{pmatrix}
     \lambda & \frac{1}{n} & 0 & \\
     0 & \lambda & \frac{1}{n} & \\
		 0 & 0 & \lambda & \\
		  & & & B	
			\\
   \end{pmatrix}P^{-1}$. $(A_{n})$ converges to $A$. However, for all $n \in \mathbb{N}^{*}$, $(X-\lambda)^{3}$ divides $\pi_{A_{n}}$. Hence, $A_{n} \notin F$ (theorem \ref{MaThm}). Thus, $A$ doesn't belong to the interior of $F$.
\\
\item If c) occurs. It exists $P \in GL_{n}(\mathbb{C})$ such that $A=P\begin{pmatrix}
    J_{\lambda} & &  \\
     & \lambda & \\
		  & & B	
			\\
   \end{pmatrix}P^{-1}$. For all $n \in \mathbb{N}^{*}$, we define $A_{n}=P\begin{pmatrix}
    \lambda & 1 &  0 & \\
    0 & \lambda &  \frac{1}{n} & \\
		0 & 0 & \lambda & & \\
		  &  &  & & B	\\
   \end{pmatrix}P^{-1}$. $(A_{n})$ converges to $A$. However, for all $n \in \mathbb{N}^{*}$, $(X-\lambda)^{3}$ divides $\pi_{A_{n}}$. Hence, $A_{n} \notin F$ (theorem \ref{MaThm}). Thus, $A$ doesn't belong to the interior of $F$.
\\
\end{itemize}

\noindent Hence, the interior of $F$ is equal to $\varphi^{-1}(\hat{F})$.
\\
\\iii) Let $A \in M_{n}(\mathbb{R})$. If $A$ is triangularizable then it exists a sequence $(A_{n})$ of diagonalizable matrices such that $(A_{n})$ converges to $A$. 
\\
\\Suppose now $A$ is not triangularizable. Thus, $\pi_{A}=\prod_{i=1}^{r} P_{i}^{n_{i}}\prod_{i=1}^{l} (X-\lambda_{i})^{m_{i}}$ with $r \geq 1$, $l \geq 0$, $\lambda_{i} \in \mathbb{R}$, $\lambda_{i} \neq \lambda_{j}$ for $i \neq j$, $n_{i}, m_{i} \geq 1$ and $P_{i}$ irreducible monic polynomials of degree 2, $P_{i} \neq P_{j}$ for $i\neq j$,. Therefore, $A$ is similar to a block-diagonal matrix in which each block has its minimal polynomial equals to $P_{i}^{n_{i}}$ or $(X-\lambda_{i})^{m_{i}}$ (by the kernel lemma). Hence, it is sufficient to consider square matrices whose minimal polynomial is the power of an irreducible monic polynomial of degree 2 (since the case of triangularizable matrices has already been considered). 
\\
\\Let $B \in M_{n}(\mathbb{R})$ such that $\pi_{B}=P^{m}$ with $P$ an irreducible monic polynomial of degree 2 and $m \leq \frac{n}{2}$. It exists $\lambda \in \mathbb{C}-\mathbb{R}$ such that $P(X)=(X-\lambda)(X-\overline{\lambda})$. By the decomposition theorem of Frobenius, $B$ is similar to a block diagonal matrix \[H=\begin{pmatrix}
   C(R_{1}) &        &   \\
         & \ddots &    \\
		     &        & C(R_{t})  \\
\end{pmatrix},\] with $R_{i}$ real polynomials verifying $R_{i}$ divides $R_{i+1}$. Since $\pi_{B}=\pi_{H}={\rm lcm}(R_{i}, 1 \leq i \leq t)$ and $\pi_{B}=P^{m}$ with $P$ irreducible, $R_{i}$ is a power of $P$. Thus, it exists $1 \leq r_{i} \leq m$ such that $R_{i}=P^{r_{i}}=(X-\lambda)^{r_{i}}(X-\overline{\lambda})^{r_{i}}$. We set $R_{i,n}=\prod_{j=1}^{r_{i}} (X-\lambda+\frac{j}{n^{2}})(X-\overline{\lambda}+\frac{j}{n^{2}})$. $(R_{i,n})$ converges to $R_{i}$. Hence, $C(R_{i,n})$ converges to $C(R_{i})$. Besides, $R_{i,n}$ verifies $(\mathcal{P})$ for all $n \in \mathbb{N}^{*}$ and $\pi_{C(R_{i,n})}=R_{i,n}$. Thus, $C(R_{i,n})$ is $X$-formable (theorem \ref{MaThm}). 
\\
\\So, it exists a sequence of $X$-formable matrices which converges to $B$.
\\
\\Hence, $G$ is dense in $M_{n}(\mathbb{R})$.
\\
\\iv) $\hat{G}$ is an open subset of $\mathbb{R}_{n}[X]$ (proposition \ref{33}). Hence, $\varphi^{-1}(\hat{G})=\{A \in M_{n}(\mathbb{R}),~\chi_{A}~{\rm verifies}~(\mathcal{P})\}$ is an open subset of $M_{n}(\mathbb{R})$. Besides, $\varphi^{-1}(\hat{G})$ is included in $G$ (by the theorem \ref{MaThm} and the theorem of Cayley-Hamilton). Hence, $\varphi^{-1}(\hat{G})$ is included in the interior of $G$.
\\
\\Let $A \in G$ such that $\chi_{A}$ doesn't verify $(\mathcal{P})$. $A$ is similar to a block-diagonal matrix in which each block has a size less than 2 (lemme \ref{21}). Besides, each triangularizable block has a Jordan normal form. Hence, $A$ is similar to a block-diagonal matrix $C$ in which each block is a Jordan block of size 1 or 2 or a square matrix of size 2 whose charateristic polynomial is an irreducible monic polynomial of degree 2. Since, $\chi_{A}$ doesn't verify $(\mathcal{P})$, it exists $\lambda$ such that $(X-\lambda)^{3}$ divides $\chi_{A}$ or it exists an irreducible monic polynomial $P$ of degree 2 such that $P^{2}$ divides $\chi_{A}$. Thus, one of the following occurs :
\begin{itemize}
\item a) $C$ contains two blocks $J_{\lambda}=\begin{pmatrix}
    \lambda & 1  \\
    0 & \lambda  \\
   \end{pmatrix}$;
\item b) $C$ contains three blocks $(\lambda)$;
\item c) $C$ contains one block $J_{\lambda}=\begin{pmatrix}
    \lambda & 1  \\
    0 & \lambda  \\
   \end{pmatrix}$ and one block $(\lambda)$;
\item d) $C$ contains two blocks $T=\begin{pmatrix}
    x & y  \\
    z & t  \\
   \end{pmatrix}$ and $O=\begin{pmatrix}
    u & v  \\
    w & h  \\
   \end{pmatrix}$ which have the same irreducible characteristic polynomial $S(X)=X^{2}+\alpha X+\beta$. In particular, $y, z, v, w \neq 0$.
	\\
\end{itemize}

\noindent We can study the first three cases in the same way as that used previously. 
\\
\\If d) occurs. It exists $P \in GL_{n}(\mathbb{R})$ such that $A=P\begin{pmatrix}
    T & &  \\
     & O & \\
		  & & B	
			\\
   \end{pmatrix}P^{-1}$. For all $n \in \mathbb{N}^{*}$, we define $A_{n}=P\begin{pmatrix}
    x & y & 0 & \frac{1}{n} & \\
    z & t & 0 & 0 & \\
		 &  & u & v & \\
		&  & w & h & \\
		  &  &  & & B	\\
   \end{pmatrix}P^{-1}$. The minimal poynomial of $Z_{n}=\begin{pmatrix}
    x & y & 0 & \frac{1}{n}  \\
    z & t & 0 & 0  \\
		0 & 0 & u & v  \\
		0 & 0 & w & h  \\
		\end{pmatrix}$ is equal to the minimal polynomial of $A_{n}$. Besides, $S$ divides the minimal polynomial of $Z_{n}$. $S(Z_{n})=\begin{pmatrix}
    0 & 0 & \frac{w}{n} & \frac{x+h+\alpha}{n}  \\
    0 & 0 & 0 & \frac{z}{n}  \\
		0 & 0 & 0 & 0  \\
		0 & 0 & 0 & 0  \\
		\end{pmatrix} \neq 0_{4}$ (since $w \neq 0$) and $S^{2}(Z_{n})=0_{4}$. Hence, $\pi_{Z_{n}} \neq S$ and $\pi_{Z_{n}}$ divides $S^{2}$. Thus, $\pi_{Z_{n}}=S^{2}$, since $S$ is irreducible. Hence, $\pi_{A_{n}}=S^{2}$. So, $\pi_{A_{n}}$ doesn't verify $(\mathcal{P})$ and $A_{n} \notin G$ (theorem \ref{MaThm}). However, $(A_{n})$ converges to $A$. Thus, $A$ doesn't belong to the interior of $G$.
\\
\\ Hence, the interior of $G$ is equal to $\varphi^{-1}(\hat{G})$.

\end{proof}

We have therefore, as announced, a dense subset of $M_{n}(\mathbb{R})$ and $M_{n}(\mathbb{C})$ and a complete characterization of the interior and of this set. We conclude by noticing a curious similarity between the diagonalizable matrices and the $X$-formable matrices. For this, we will use the following notation : a polynomial verifies the property $(\mathcal{Q})$ if it is a product of distinct linear factors over $\mathbb{K}$. Let $A \in M_{n}(\mathbb{K})$ with $\mathbb{K}=\mathbb{R}, \mathbb{C}$, we have :

\begin{itemize}
\item $A$ is diagonalizable if and only if $\pi_{A}$ verifies $(\mathcal{Q})$; $A$ belongs to the interior of the set of diagonalizable matrices if and only if $\chi_{A}$ verifies $(\mathcal{Q})$;
\item $A$ is $X$-formable if and only if $\pi_{A}$ verifies $(\mathcal{P})$; $A$ belongs to the interior of the set of $X$-formable matrices if and only if $\chi_{A}$ verifies $(\mathcal{P})$.
\end{itemize}

\section{Some open problems}

\indent We collect here some open problems related to $X$-formable matrices. The firt problem is related to theorem \ref{MaThmBis}. Indeed, we know some topological properties of the set of $X$-formable matrices. Therefore, it is natural to want to look for other. 

\begin{pro}
\label{41}

Find other topological properties of the set of $X$-formable matrices over $\mathbb{R}$ and $\mathbb{C}$.

\end{pro}

We have a certain number of results concerning the maximum dimension of a vector subspace all of whose elements verify a given property (such as Gerstenhaber's theorem, see \cite{Ge}). This naturally leads to the following problem :

\begin{pro}
\label{42}

What is the maximal dimension of a vector subspace of $M_{n}(\mathbb{K})$ containing only $X$-formable elements ?

\end{pro}

\noindent In the case of diagonalizable matrices we have the following result :

\begin{thm}[Klar\`{e}s's criterion, \cite{MM} p 125]
\label{43}

Let $B \in M_{n}(\mathbb{K})$ with $\mathbb{K}$ an algebraically closed field. We set $Ad_{B} : M \in M_{n}(\mathbb{K}) \longmapsto BM-MB$. $B$ is diagonalizable if and only if $Ker(Ad_{B})=Ker(Ad_{B}^{2})$.

\end{thm}

\noindent This leads to the formulation of the problem below :

\begin{pro}
\label{44}

Can we find a ``similar" criterion for $X$-formable matrices ?

\end{pro}

\end{document}